\documentclass{article}

\usepackage{graphicx}
\usepackage{indentfirst}
\usepackage{amsmath,amsfonts,amsthm,amssymb,amsrefs}
\usepackage{mathrsfs,stmaryrd}
\usepackage{xcolor}
\usepackage[all]{xy}
\usepackage{hyperref}

\usepackage{relsize}

\DeclareMathAlphabet{\mathsfsl}{OT1}{cmss}{m}{sl}

\newtheorem{thm}{Theorem}[section]
\newtheorem{lem}[thm]{Lemma}
\newtheorem{cor}[thm]{Corollary}
\newtheorem{prop}[thm]{Proposition}

\theoremstyle{definition}
\newtheorem{defn}[thm]{Definition}

\newtheorem{rem}[thm]{Remark}

\begin{document}

\title{Knot Floer homology and fixed points}

\author{{\Large Yi NI}\\{\normalsize Department of Mathematics, Caltech, MC 253-37}\\
{\normalsize 1200 E California Blvd, Pasadena, CA
91125}\\{\small\it Emai\/l\/:\quad\rm yini@caltech.edu}}

\date{}
\maketitle

\begin{abstract}
If $K$ is a fibered knot in a closed, oriented $3$--manifold $Y$ with fiber $F$, and $\widehat{HFK}(Y,K,[F], g(F)-1;\mathbb Z/2\mathbb Z)$ has rank $r$, then the monodromy of $K$ is freely isotopic to a diffeomorphism with at most $r-1$ fixed points.
This generalizes earlier work of Baldwin--Hu--Sivek and Ni.

We also clarify a misleading formula in Cotton-Clay's computation of the symplectic Floer homology of mapping classes of surfaces.
\end{abstract}

\section{Introduction}

Knot Floer homology, defined by Ozsv\'ath--Szab\'o \cite{OSzKnot} and Rasmussen \cite{RasThesis}, is a powerful knot invariant with many applications. One important property of knot Floer homology is that the Seifert genus $g(K)$ of a knot $K\subset S^3$ can be read from the knot Floer homology \cite{OSzGenus}:
\[
g(K)=\max\{a|\:\widehat{HFK}(S^3,K,a)\ne0\}.
\]
It is also known that $K$ is fibered if and only if $\widehat{HFK}(S^3,K,g(K))\cong\mathbb Z$ \cite{Gh,NiFibred}. 

It is an interesting question to ask what topological property about $K$ can be deduced from other summands of $\widehat{HFK}(S^3,K)$. In \cite{BHS}, Baldwin--Hu--Sivek proved that if a hyperbolic knot $K\subset S^3$ has the same knot Floer homology as the torus knot $T(5,2)$, then the monodromy of $K$ is freely isotopic to a pseudo-Anosov map without fixed point. It is observed in \cite{NiFPFree} that their argument can be applied to hyperbolic fibered knot in any closed $3$--manifold. In fact, the argument in \cite{NiFPFree} implies the following upper bound on the number of fixed points.

\begin{thm}\label{thm:FPhyp}
Let $Y$ be a closed, oriented $3$--manifold, and $K\subset Y$ be a hyperbolic fibered knot with
fiber $F$ and monodromy $\varphi$. If \[\mathrm{rank}\widehat{HFK}(Y,K,[F],g(F)-1)=r,\] then $\varphi$ is freely isotopic to a pseudo-Anosov map with at most $r-1$ fixed points. 
\end{thm}

The main theorem in this paper removes the hyperbolic condition in the above result. For technical reasons, we will use coefficients in $\mathbb F=\mathbb Z/2\mathbb Z$.

\begin{thm}\label{thm:FPbound}
Let $Y$ be a closed, oriented $3$--manifold, and $K\subset Y$ be a fibered knot with
fiber $F$ and monodromy $\varphi$. If \[\mathrm{rank}\widehat{HFK}(Y,K,[F],g(F)-1;\mathbb F)=r,\] then $\varphi$ is freely isotopic to a diffeomorphism with at most $r-1$ fixed points. 
\end{thm}

The proof of Theorem~\ref{thm:FPbound} uses the same strategy as in \cite{BHS,NiFPFree}. 
Ghiggini and Spano \cite{GS} proved a similar result for certain area-preserving diffeomorphisms using a completely different framework. See Subsection~\ref{subsect:GS} for comparisons of our result with Ghiggini--Spano's work.

One important ingredient of the proof is Cotton-Clay's computation of the symplectic Floer homology of a mapping class \cite{CC}. For each mapping class of a closed surface, Cotton-Clay constructed a ``perturbed standard form map'' which is an area-preserving diffeomorphism. (A lot of such constructions have already been carried out in the work of Jiang--Guo \cite{JiangGuo} and Gautschi \cite{Gautschi}.) He then computed the symplectic Floer homology of this perturbed standard form map. 

There was, however, a gap in the statement of his computation. Below, we will state the theorem in a less confusing way. The meaning of the terms in the formula, as well as the explanation of the mistake, will be given in Section~\ref{sect:SFH}. This gap does not affect the argument in \cite{BHS} and \cite{NiFPFree}.

\begin{thm}\label{thm:SFHcorrect}
Let $\widehat{\phi}:\Sigma\to\Sigma$ be a perturbed standard form map in a reducible mapping class $h\in\Gamma$. Then
\begin{align}
HF_*^{\mathrm{symp}}(h)=HF_*^{\mathrm{symp}}(\widehat{\phi})\cong& H_{*(\mathrm{mod}\;2)}(\Sigma_a,\partial_-\Sigma_a;\mathbb F)\nonumber\\
&\oplus\bigoplus_p\big(H_{*(\mathrm{mod}\;2)}(\Sigma_{b,p}^{\circ},\partial_-\Sigma_{b,p};\mathbb F)\oplus \mathbb F^{(p-1)|\pi_0(\Sigma_{b,p})|}\big)\nonumber\\
&\oplus\bigoplus_q\big(H_{*(\mathrm{mod}\;2)}(\Sigma_{c,q},\partial_-\Sigma_{c,q};\mathbb F)\oplus \mathbb F^{q|\pi_0(\Sigma_{c,q})|}\big)\nonumber\\
&\oplus\mathbb F^{\Lambda(\widehat{\phi}|_{\Sigma_1})}\oplus\mathbb F^{2n_{f}}\oplus CF_*^{\mathrm{symp}}(\widehat{\phi}|_{\Sigma_2}).\label{eq:SFH}
\end{align}
\end{thm}

The difference between (\ref{eq:SFH}) and the formula in \cite[Theorem~4.16]{CC} is that there is an extra term $\mathbb F^{2n_{f}}$ in (\ref{eq:SFH}), where $n_f$ is the number of annuli on which the restriction of $\widehat{\phi}$ is a flip-twist map. See also Remark~\ref{rem:CC} for another interpretation of \cite[Theorem~4.16]{CC} which yields the correct result.


\subsection{Comparisons with Ghiggini--Spano's work}\label{subsect:GS}

Both our work and Ghiggini--Spano's \cite{GS} prove the isomorphism between the next-to-top term of the knot Floer homology of a fibered knot $K\subset Y$ and some versions of the symplectic Floer homology of mapping classes of surfaces. In our paper, we use the symplectic Floer homology of mapping classes of a closed surface $\Sigma$. The closed surface $\Sigma$ and the mapping class are constructed from the Seifert surface $F$ and the monodromy of $K$. In Ghiggini--Spano's work, the authors defined the symplectic Floer homology of mapping classes of a surface with boundary, and showed that for the Seifert surface $F$ and the monodromy of a fibered knot $K$, the symplectic Floer homology is isomorphic to the next-to-top term of the knot Floer homology of $K$ \cite[Theorem~1.1]{GS}. 

After establishing the isomorphism, we then use Cotton-Clay's work \cite{CC} to show Theorem~\ref{thm:FPbound}, which says that $r-1$ is an {\it upper bound} to the minimal number of fixed points among {\it all} diffeomorphisms in the mapping class of the monodromy. On the other hand, Ghiggini--Spano \cite{GS} only considered the minimal number of fixed points of {\it area-preserving non-degenerate} diffeomorphisms in this
mapping class of the monodromy. In this case, they get the {\it exact value} of this minimal number of fixed points \cite[Theorem~1.2]{GS}, which is certainly an upper bound to the minimal number of fixed points among all diffeomorphisms in this mapping class. (Or in their original language, it is a {\it sharp lower bound} to the number of fixed points of area-preserving non-degenerate diffeomorphisms in this
mapping class.)
For their purpose, they also use another work of Cotton-Clay \cite{CC2}, and they need the technical condition that either the ambient $3$--manifold $Y$ is a rational homology sphere or the mapping class is irreducible.

The two kinds of minimal numbers of fixed points considered in our work and Ghiggini--Spano's work \cite{GS} agree for a lot of mapping classes including the pseudo-Anosov ones.

\subsection{Organizations}
This paper is organized as follows. In Section~\ref{sect:Nielsen}, we review some results in Nielsen fixed point theory. In Section~\ref{sect:SFH}, we review Cotton-Clay's work on symplectic Floer homology of mapping classes, and explain how we get Theorem~\ref{thm:SFHcorrect}. In Section~\ref{sect:Main}, we use the zero-surgery formula in Heegaard Floer homology to prove Theorem~\ref{thm:FPbound}.

\vspace{5pt}\noindent{\bf Acknowledgements.}\quad  The author was
partially supported by NSF grant number DMS-1811900. We are grateful to John Baldwin, Andrew Cotton-Clay, Ko Honda, and Steven Sivek for helpful comments. The author is particularly grateful to John Baldwin for pointing out a gap in an earlier version of this paper and suggesting a fix.


\section{Nielsen theory of surfaces}\label{sect:Nielsen}

In this section, we will give a brief introduction to Nielsen theory following Jiang \cite{Jiang}. Then we will review Jiang--Guo's work \cite{JiangGuo} on the Nielsen problem for surface diffeomorphisms. 

\begin{defn}
Let $X$ be a topological space, $f: X\to X$ be a continuous map. We say two fixed points $a_0,a_1\in\mathrm{Fix}(f)$ are {\it Nielsen equivalent}, if there exists a path $\gamma$ from $a_0$ to $a_1$, such that $\gamma$ is homotopic to $f(\gamma)$ rel $\partial$. A {\it Nielsen class} is the collection of all fixed points which are Nielsen equivalent to a given fixed point.
\end{defn}

If $\mathcal C$ is a Nielsen class, then $\mathcal C$ is an open set in $\mathrm{Fix}(f)$ \cite[Theorem~1.12]{Jiang}.

\begin{defn}
Let $U\subset \mathbb R^n$ be an open set, $f: U\to \mathbb R^n$ be a continuous map, $a\in \mathrm{Fix}(f)$ be an isolated fixed point. Let $B_{\varepsilon}(a)\subset U$ be a small open ball about $a$ satisifying $B_{\varepsilon}(a)\cap \mathrm{Fix}(f)=\{a\}$.
Then the {\it fixed point index} of $a$, denoted by $\mathrm{index}(f,a)$, is the degree of the map $\varphi:\partial B_{\varepsilon}(a)\to S^{n-1}$ defined by
\[
x\mapsto\frac{x-f(x)}{||x-f(x)||}.
\]
\end{defn}

The above definition can be extended to the case when $a$ is an isolated fixed point of a map $f: X\to X$ with $X$ being a manifold, since the nature of the definition is local. Moreover, for any open set $U\subset X$ such that $U\cap \mathrm{Fix}(f)$ is compact, one can define the {\it fixed point index} $\mathrm{index}(f,U)$. One way to define the index is to approximate $f$ with a generic smooth map $g$ in $U$ with only isolated fixed points, and define $\mathrm{index}(f,U)$ to be the sum of the indices of fixed points of $g$. 

Now let $X$ be a compact manifold, $f:X\to X$ be a map. Given a Nielsen class $\mathcal C$, since $\mathcal C$ is an isolated set of fixed points, we can choose an open set $U\subset X$ such that $U\cap \mathrm{Fix}(f)=\mathcal C$. We can define the {\it fixed point index} \[\mathrm{index}(f,\mathcal C):=\mathrm{index}(f,U).\]

\begin{thm}[Lefschetz--Hopf]\label{thm:LefHopf}
The total fixed point index $\mathrm{index}(f,X)$ is equal to the Lefschetz number $\Lambda(f)$.
\end{thm}

\begin{defn}
A Nielsen class $\mathcal C$ is {\it essential} if $\mathrm{index}(f,\mathcal C)\ne0$.
The {\it Nielsen number} $N(f)$ of a map $f$ is the number of essential Nielsen classes of $f$.
\end{defn}

The Nielsen number $N(f)$ is a homotopy invariant of the map $f$. It gives a lower bound to the number of fixed points of any map in the homotopy class of $f$. In many cases, this lower bound is sharp. The following theorem of Jiang--Guo \cite{JiangGuo} establishes the sharpness in the case of surface diffeomorphisms in a given isotopy class. (The original theorem is stated for possibly non-orientable surfaces.)

\begin{thm}[Jiang--Guo]\label{thm:JG}
Let $\Sigma$ be a compact, oriented, connected surface, $f:\Sigma\to\Sigma$ be an orientation-preserving homeomorphism. Then $f$ is isotopic to a diffeomorphism $g$ which has exactly $N(f)$ fixed points.
\end{thm}


\section{The symplectic Floer homology of a mapping class}\label{sect:SFH}

In this section, we will review Cotton-Clay's computation of the symplectic Floer homology of a mapping class. We will work over $\mathbb F=\mathbb Z/2\mathbb Z$.

Let $\Sigma$ be a closed, oriented, connected surface with an area form $\omega$ and $g(\Sigma)\ge2$. Let $\phi$ be an area-preserving diffeomorphism of $\Sigma$, such that all fixed points of $\phi$ are non-degenerate, and $\phi$ satisfies a monotonicity condition. The existence of such diffeomorphisms was proved in \cite{Seidel}, see \cite[Lemma~7]{Seidel} and the paragraph before it. The symplectic Floer chain complex $CF_*^{\mathrm{symp}}(\phi)$ is an $\mathbb Z/2\mathbb Z$--graded vector space over $\mathbb F$, which is freely generated by the fixed points of $\phi$. The differential $\partial$ counts holomorphic disks satisfying certain conditions. 
A fixed point $y$ appears in $\partial x$ for another fixed point $x$, only if $x$ and $y$ are Nielsen equivalent. So
$CF_*^{\mathrm{symp}}(\phi)$ naturally splits as a direct sum over all Nielsen classes.

Seidel \cite{Seidel} proved that $HF_*^{\mathrm{symp}}(\phi)$ only depends on the mapping class $h$ of $\phi$, so one can define
\[
HF_*^{\mathrm{symp}}(h)=HF_*^{\mathrm{symp}}(\phi)
\]
to be the symplectic Floer homology of the mapping class $h$.
The symplectic Floer homology of mapping classes was extensively studied in \cite{Gautschi,CC}.

By Thurston's classification of surface automorphisms \cite{ThurstonSurface}, we may find a diffeomorphism $\phi$ in a mapping class $h$, which is in the {\it standard form} in the following sense:
There exists a finite union $N$ of disjoint non-contractible annuli, such that $\phi(N)=N$. Moreover, the following conditions are satisfied:

\ 

\quad(1) For any component $A$ of $N$, there exists a positive integer $\ell$, such that $\phi^{\ell}(A)=A$ and $\phi^{\ell}|_A$ is a {\it twist map} or a {\it flip-twist map}. A twist map on $[0,1]\times S^1$ has the form
\[
(q,p)\mapsto(q,p-f(q)),
\]
and a flip-twist map on $[0,1]\times S^1$ has the form
\[
(q,p)\mapsto(1-q,-p+f(q)),
\]
where $f:[0,1]\to\mathbb R$ is a strictly monotonic smooth map. The (flip-)twist map is {\it positive} or {\it negative} if $f$ is increasing or decreasing, respectively.

\quad(2) In the first condition, if $\ell=1$ and $\phi|_A$ is a twist map, then $\phi$ has no fixed point in $\mathrm{int}(A)$. That is, $\mathrm{im}(f)\subset[0,1]$. We also require that parallel twist regions are twisted in the same direction.

\quad(3) For any component $S$ of $\Sigma\setminus\mathrm{int}(N)$, there exists a positive integer $\ell$, such that $\phi^{\ell}(S)=S$ and $\phi^{\ell}|_S$ is either periodic or pseudo-Anosov. 

\ 

There are 3 types of fixed points of $\phi$, listed in \cite[Subsection~4.3]{CC}. Type I consists of points in components $S$ of $\Sigma\setminus \mathrm{int}(N)$ which are fixed by $\phi$ pointwise. Such components are called {\it fixed components}. This type is further divided into Type Ia ($\chi(S)<0$) and Ib ($\chi(S)=0$). Type II consists of fixed points in periodic components (Type IIa) and flip-twist regions (Type IIb). Type III consists of fixed points in pseudo-Anosov components. This type is further divided into 4 sub-types: IIIa (fixed points not associated with any singular points or punctures), IIIb (fixed points associated with unrotated singular points), IIIc (fixed points associated with rotated singular points), IIId (fixed points associated with unrotated punctures). 

In \cite[Subsection~4.5]{CC}, $\phi$ is further perturbed to an area-preserving diffeomorphism $\widehat{\phi}$ for which one will compute the symplectic Floer homology. 

The subsurface $\Sigma\setminus\mathrm{int}(N)$ can be divided into three parts: Let $\Sigma_0$ be the collection of fixed components, $\Sigma_1$ be the collection of (non-fixed) periodic components, and $\Sigma_2$ be the collection of pseudo-Anosov components with punctures.

There are two types of boundary components of $\Sigma_0$ defined in \cite[Subsection~4.5]{CC}: $\partial_{\pm}\Sigma_0$. 

Let $\Sigma_a$ be the collection of fixed components which do not meet any pseudo-Anosov components. Let $\Sigma_{b,p}$ be the collection of fixed components which meet one pseudo-Anosov component at a boundary with $p$ prongs.  Let $\Sigma^{\circ}_{b,p}$ be the collection of components of $\Sigma_{b,p}$ with each component punctured once. Let $\Sigma_{c,q}$ be the collection of fixed components which meet at least two pseudo-Anosov components such that the total number of prongs over all the boundaries is $q$.

Let $\Lambda(\widehat{\phi}|_{\Sigma_1})$ be the Lefschetz number of $\widehat{\phi}|_{\Sigma_1}$. Let $CF_*^{\mathrm{symp}}(\widehat{\phi}|_{\Sigma_2})$ be the symplectic Floer chain complex for $\widehat{\phi}|_{\Sigma_2}$ on the component $\Sigma_2$. More precisely, $CF_*^{\mathrm{symp}}(\widehat{\phi}|_{\Sigma_2})$ is freely generated by the fixed points of $\widehat{\phi}|_{\Sigma_2}$, except that when a
pseudo-Anosov component abuts a fixed component, we include the boundary fixed
points as part of the fixed component and not as part of the pseudo-Anosov component.

In \cite[Theorem~4.16]{CC}, Cotton-Clay stated the following formula for the symplectic Floer homology of a reducible mapping class. Let $\widehat{\phi}:\Sigma\to\Sigma$ be a perturbed standard form map in a reducible mapping class $h$. Then
\begin{align}
HF_*^{\mathrm{symp}}(h)=HF_*^{\mathrm{symp}}(\widehat{\phi})\cong& H_{*(\mathrm{mod}\;2)}(\Sigma_a,\partial_-\Sigma_a;\mathbb F)\nonumber\\
&\oplus\bigoplus_p\big(H_{*(\mathrm{mod}\;2)}(\Sigma_{b,p}^{\circ},\partial_-\Sigma_{b,p};\mathbb F)\oplus \mathbb F^{(p-1)|\pi_0(\Sigma_{b,p})|}\big)\nonumber\\
&\oplus\bigoplus_q\big(H_{*(\mathrm{mod}\;2)}(\Sigma_{c,q},\partial_-\Sigma_{c,q};\mathbb F)\oplus \mathbb F^{q|\pi_0(\Sigma_{c,q})|}\big)\nonumber\\
&\oplus\mathbb F^{\Lambda(\widehat{\phi}|_{\Sigma_1})}\oplus CF_*^{\mathrm{symp}}(\widehat{\phi}|_{\Sigma_2}).\label{eq:CC}
\end{align}

However, there is an omission in the above formula. It happened in \cite[Lemma~4.15]{CC}, where the author stated that ``... the Floer homology chain complex $(CF_*(\phi),\partial_{J_t})$ splits into a sum of chain
complexes $(\mathcal C_i,\partial_i)$ for each component of $\Sigma\setminus N$.'' The problem is, if $\phi|_{A}$ is a flip-twist map for some component $A$ of $N$, there will be two Type IIb fixed points in $A$, and each of these two fixed points is the only fixed point in its Nielsen class as argued in \cite[Section~4.3]{CC}. So there is a contribution of a rank $2$ summand from $A\subset N$, which is not included in the statements of \cite[Lemma~4.15 and Theorem~4.16]{CC}. 

In fact, in the statement of \cite[Theorem~1.6]{CC}, the author stated that ``$HF_*(h)$ splits into summands for each component of $\Sigma\setminus C$'', where $C=\partial N$ is a collection of simple closed curves preserved by $h$. So the contribution of the flip-twist annuli is implicitly included in this statement.



To get the correct formula, we should include the flip-twist annuli. Let $n_f$ be the number of annuli on which the restriction of $\widehat{\phi}$ is a flip-twist map. Then there is a summand $\mathbb F^{2n_f}$ in $CF_*(\widehat{\phi})$, on which the differential is zero. So $HF_*(\widehat{\phi})$ should have such a summand. Other summands of $HF_*(\widehat{\phi})$ can be obtained in the same way as in \cite{CC}. This finishes the proof of Theorem~\ref{thm:SFHcorrect}.

It is not hard to see 
\[
2n_f=\Lambda(\widehat{\phi}|_N).
\]
So the extra summand may also be written as $\mathbb F^{\Lambda(\widehat{\phi}|_N)}$.

\begin{rem}\label{rem:CC}
According to Cotton-Clay \cite{CCp}, one can adjust the statements of \cite[Lemma~4.15 and Theorem~4.16]{CC} so that (\ref{eq:CC}) still gives the correct result. One way is to define $\Sigma_1$ to be the disjoint union of (non-fixed) periodic regions and the flip-twist annuli. The other way is, when one defines the standard form map, if $\phi|_A$ is a flip-twist map, one decomposes $A$ as the union of a $2$--periodic annuli and two twist annuli, hence $\Sigma_1$ will include a $2$--periodic annulus for each flip-twist annulus. Nonetheless, in this paper we still choose (\ref{eq:SFH}) as the formula for symplectic Floer homology since it is more direct.
\end{rem}

\begin{cor}\label{cor:SympFPbound}
Let $\Sigma$ be a closed, oriented, connected surface, $\phi:\Sigma\to\Sigma$ be an orientation-preserving homeomorphism, and $h$ be the mapping class of $\phi$. Then
\[
N(\phi)\le\mathrm{rank}HF^{\mathrm{symp}}_*(h).
\]
\end{cor}
\begin{proof}
Let $\widehat{\phi}$ be a standard form map isotopic to $\phi$.
For the terms in the last row of (\ref{eq:SFH}), the generators correspond to Type II and Type III fixed points, except that when a
pseudo-Anosov component abuts a fixed component, we include the Type IIId boundary fixed
points as part of the fixed component and not as part of the pseudo-Anosov component. As summarized in the proof of \cite[Lemma~4.15]{CC}, for Type IIa, Type IIb, Type IIIa and Type IIIc fixed points, each fixed point is the only fixed point in its Nielsen class.
Moreover, for Type IIIb fixed points and Type IIId fixed points which do not abut a Type Ia component, the Nielsen class consists of fixed points of the same index. It is clear that the Nielsen classes corresponding to Type II and Type III fixed points are all essential. 
So the rank of the last row of (\ref{eq:SFH}) is greater than or equal to the number of corresponding essential Nielsen classes.

Let $S$ be a component of $\Sigma_0$, and let $C$ be a {possibly empty} collection of boundary components of $S$. As in Theorem~\ref{thm:SFHcorrect}, let $S^{\circ}$ be $S$ with an interior point removed.  Let $U$ be an open tubular neighborhood of $S$, then 
\[
\mathrm{index}(\widehat{\phi},U)=\chi(S).
\]
When $\chi(S)=0$, $S$ itself is the collection of all fixed points in a Nielsen class by \cite[Corollary~4.9]{CC}, so this Nielsen class is inessential.
When $\chi(S)<0$, it is elementary to check 
\[
\mathrm{rank}\:H_*(S,C;\mathbb F)>0 \text{ and }\mathrm{rank}\:H_*(S^{\circ},C;\mathbb F)>0.
\]
So $S$ contributes at least $1$ to the total rank of (\ref{eq:SFH}), while it contributes at most $1$ to $N(\widehat{\phi})$.

Combining the above analysis, we get our inequality.
\end{proof}

\section{Proof of the main theorem}\label{sect:Main}

In this section, we will prove Theorem~\ref{thm:FPbound} using the argument in \cite{NiFPFree} and \cite{BNS}. We will use $\mathbb F$--coefficients for Heegaard Floer homology. 

When $\varphi=\mathrm{id}$, $Y=\#^{2g(F)}S^1\times S^2$ and $K$ is the Borromean knot. By \cite{OSzKnot}, $\mathrm{rank}\widehat{HFK}(Y,K,[F],g(F)-1;\mathbb F)=2g(F)$. On the other hand, it is elementary to freely isotope $\mathrm{id}$ to a map with exactly $|\chi(F)|=2g(F)-1$ fixed points. So the conclusion of Theorem~\ref{thm:FPbound} holds true in this case.

From now on, we assume $\varphi$ is not isotopic to $\mathrm{id}$ rel $\partial$.

Given a closed, oriented, connected $3$--manifold $Y$ and a subset $\mathfrak S\subset \mathrm{Spin}^c(Y)$, let
\[
CF^+(Y,\mathfrak S)=\bigoplus_{\mathfrak s\in \mathfrak S}CF^+(Y,\mathfrak s).
\]
Similarly, we can define the corresponding summands of various Heegaard Floer chain complexes and homology groups.

As in \cite[Lemma~4.1]{NiFPFree},
let $L\subset Z=S^1\times S^2$ be a hyperbolic fibered knot with fiber $G$, such that the $\mathrm{Spin}^c$ structure $\mathfrak r$ of the open book decomposition with binding $L$ and page $G$ is non-torsion. Let $L'$ be the $(3\ell+1,3)$--cable of $L$
for a sufficiently large integer $\ell$. (Our convention is that the cable knot is homologous to $3\ell+1$ meridians plus $3$ longitudes in the boundary of a tubular neighborhood of $L$.) Then $L'$ has a Seifert surface $E$, which is the union of three copies of $G$ and a genus--$(3\ell+1)$ surface $T$ with $|\partial T|=4$. In fact, let $V$ be a regular neighborhood of $L$, then $T$ is a fiber of a fibration of $V\setminus L'$. (One can think of $T$ as the Seifert surface of the $(3\ell+1,3)$ torus knot with three disks removed.) Then $E$ is the fiber of a fibration of $Z\setminus L'$. 

The following lemma follows from Hedden \cite{Hedden}. See \cite[Lemma~3.2]{BNS}.

\begin{lem}\label{lem:Cable}
We have
\[
\widehat{HFK}(Z,L',[E],g(L'))\cong\widehat{HFK}(Z,L',[E],g(L')-1)\cong\mathbb F,
\]
and
\[
\widehat{HFK}(Z,L',[E],g(L')-2)=0. 
\]
Moreover, both $\widehat{HFK}(Z,L',[E],g(L'))$ and $\widehat{HFK}(Z,L',[E],g(L')-1)$ are supported in the $\mathrm{Spin}^c$ structure $\mathfrak r$.
\end{lem}

Let $(Y,K,F)$ be as in the statement of Theorem~\ref{thm:FPbound}. Let $\overline{L'}\subset Z$ be the mirror image of $L'$, and $\overline E$ be the mirror image of $E$.
We will consider the connected sum $Y^{\#}=Y\#Z$, the knot 
 $K^{\#}=K\#\overline{L'}\subset Y\#Z$ and its Seifert surface $F^{\#}=F\natural \overline E$.
 
Below we will use the concepts of left-veering and right-veering 
diffeomorphisms, which can be found in \cite{HKMRightVeering}.

As in the statement of Theorem~\ref{thm:FPbound}, let $\varphi: F\to F$ be the monodromy of $K\subset Y$. By the assumption we made earlier in this section, $\varphi$ is not isotopic to $\mathrm{id}$ rel $\partial$. So $\varphi$ cannot be both left-veering and right-veering.
Without loss of generality, we may assume $\varphi$ is not left-veering.

 Let $\psi:E\to E$ be the monodromy of $L'$, then $\psi$ permutes the three copies of $G$, and its restriction on $T$ is a periodic map of period $9\ell+3$ which has no fixed point. Hence $\psi$ has no fixed point.
 Let $\sigma:F^{\#}\to F^{\#}$ be the monodromy of $K^{\#}\subset Y^{\#}$, and $g^{\#}$ be the genus of the fiber $F^{\#}$ of this knot complement.

\begin{lem}\label{lem:NoVeering}
Under the assumption that $\varphi$ is not left-veering, we can conclude that
 $\sigma$ is neither left-veering nor right-veering.
\end{lem}
\begin{proof}
Using \cite[Propositions~2.5~and~4.2]{KR}, we see that the monodromy $\psi$ of $L'$ is right-veering, hence $\sigma|_{\overline{E}}$ is left-veering. Since $\varphi=\sigma|_{F}$ is not left-veering,
$\sigma$ is neither left-veering nor right-veering.
\end{proof}

Every $\mathrm{Spin}^c$ structure $\mathfrak s\in \mathrm{Spin}^c(Y^{\#})$ is the connected sum of a $\mathrm{Spin}^c$ structure over $Y$ and a $\mathrm{Spin}^c$ structure over $Z$. Let $p(\mathfrak s)$ be the $\mathrm{Spin}^c$ structure over $Z$.
Define
\begin{align*}
\mathfrak S^{\bullet}&=\{
\mathfrak s\in \mathrm{Spin}^c(Y^{\#})| p(\mathfrak s)=\mathfrak r
\},\\
\mathfrak  S^{!}&=\{
\mathfrak s\in \mathrm{Spin}^c(Y^{\#})| p(\mathfrak s)\ne\mathfrak r
\}.
\end{align*}

Now the chain complex $CFK^{\infty}(Y^{\#},K^{\#},[F^{\#}])$ is the direct sum of
two subcomplexes 
\begin{align*}
C^{\bullet}&=CFK^{\infty}(Y^{\#},K^{\#},\mathfrak S^{\bullet},[F^{\#}]),\\
C^{!}&=CFK^{\infty}(Y^{\#},K^{\#},\mathfrak S^{!},[F^{\#}]).
\end{align*}

Let $\widehat{F^{\#}}\subset (Y^{\#})_0(K^{\#})$ be the closed surface obtained by capping off $\partial F^{\#}$ with a disk.
Given $\mathfrak s\in \mathrm{Spin}^c(Y^{\#})$ and $k\in\mathbb Z$,
 let $\mathfrak t_k=\mathfrak t_k(\mathfrak s)\in\mathrm{Spin}^c((Y^{\#})_0(K^{\#}))$ be the Spin$^c$ structure satisfying
\[
\mathfrak t_k|(Y^{\#}\setminus (K^{\#}))=\mathfrak s|(Y^{\#}\setminus (K^{\#})),\quad \langle c_1(\mathfrak t_k),[\widehat{F^{\#}}]\rangle=2k.
\]
We will often consider $(Y^{\#})_n(K^{\#})$, the $n$--surgery on $K^{\#}$ for a given integer $n\gg0$. In this case,
let
$[\mathfrak t_k]$ be the set of all Spin$^c$ structures $\mathfrak t$ on $Y^{\#}_0(K^{\#})$ satisfying that 
\[
\mathfrak t|(Y^{\#}\setminus K^{\#})=\mathfrak s|(Y^{\#}\setminus K^{\#}),\quad \langle c_1(\mathfrak t)-c_1(\mathfrak t_k),\widehat{F^{\#}}\rangle=2nm,\quad \text{for some } m\in\mathbb Z.
\]
Define
\[
\mathfrak T^{\bullet}=\{\mathfrak t_k(\mathfrak s)|\mathfrak s\in\mathfrak S^{\bullet}\},\quad
[\mathfrak T^{\bullet}]=\cup_{\mathfrak s\in\mathfrak S^{\bullet}}[\mathfrak t_k(\mathfrak s)],
\]
and define $\mathfrak T^!,[\mathfrak T^!]$ similarly.

The following proposition is an analogue of \cite[Proposition~3.1]{NiFPFree}.

\begin{prop}\label{prop:ZeroSurg}
Under the assumption that $\varphi$ is not left-veering,
we have
\[
\mathrm{rank}\:HF^+(Y^{\#}_0(K^{\#}),\mathfrak T^{\bullet}_{g^{\#}-2})=\mathrm{rank}\widehat{HFK}(Y^{\#},K^{\#},\mathfrak S^{\bullet},[F],g^{\#}-1)-2.
\]
\end{prop}
\begin{proof}
In \cite[Section~9]{OSzAnn2},
Ozsv\'ath and Szab\'o proved an exact sequence 
\begin{equation}\label{eq:ExTriangle}
\cdots\to HF^+(Y^{\#}_0(K^{\#}),[\mathfrak t_k])\to HF^+(Y^{\#}_n(K^{\#}),\mathfrak s_k)\to HF^+(Y^{\#},\mathfrak s)\to\cdots,
\end{equation}
where $\mathfrak s\in \mathrm{Spin}^c(Y^{\#})$, $\mathfrak s_k$ is a certain Spin$^c$ structure on $Y^{\#}_n(K^{\#})$.
When $n\ge g^{\#}+|k|$, by the adjunction inequality \cite[Theorem~7.1]{OSzAnn2}, $\mathfrak t_k$ is the only Spin$^c$ structure in $[\mathfrak t_k]$ which possibly supports a nonzero summand of $HF^+(Y^{\#}_0(K^{\#}))$. Hence
\begin{equation}\label{eq:OnlySpinc}
HF^+(Y^{\#}_0(K^{\#}),[\mathfrak t_k])=HF^+(Y^{\#}_0(K^{\#}),\mathfrak t_k).
\end{equation}

For simplicity, we will write $C=C^{\bullet}=CFK^{\infty}(Y^{\#},K^{\#},\mathfrak S^{\bullet},[F^{\#}])$ for the rest of the proof.
It follows from \cite[Theorem 4.4]{OSzKnot} that when $n\gg0$, we have
\begin{equation}\label{eq:LargeN}
HF^+(Y^{\#}_n(K^{\#}),\mathfrak S^{\bullet}_k)\cong H_*(A_k^+):=H_*(C\{i\ge0\text{ or }j\ge k\}),
\end{equation} 
where $\mathfrak S^{\bullet}_k=\{\mathfrak s_k|\mathfrak s\in \mathfrak S^{\bullet}\}$.

Since $\mathfrak r$ is nontorsion in $S^1\times S^2$,
$HF^+(Y^{\#},\mathfrak S^{\bullet})=0$. Combining (\ref{eq:ExTriangle}), (\ref{eq:OnlySpinc}), and (\ref{eq:LargeN}),  we get
\[
HF^+(Y^{\#}_0(K^{\#}),\mathfrak T^{\bullet}_{g-2})\cong H_*(A_{g-2}^+).
\]

By the exact triangle
\[
\xymatrix{
H_*(A^+_{g-2})\ar[r]&H_*(C\{i\ge0\})\ar[ld]\\ 
H_*(C\{i<0, j\ge g-2\})\ar[u]&  
}
\]
and the fact that $H_*(C\{i\ge0\})\cong HF^+(Y^{\#},\mathfrak S^{\bullet})=0$, we get that
\[H_*(C\{i<0, j\ge g-2\})\cong H_*(A^+_{g-2}).\]
It follows that 
\[
HF^+(Y^{\#}_0(K^{\#}),\mathfrak T^{\bullet}_{g-2})\cong H_*(C\{i<0, j\ge g-2\}).
\]

Using Lemma~\ref{lem:NoVeering} and \cite[Lemma~3.2]{NiFPFree}, we get our conclusion.
\end{proof}

The following proposition was proved in \cite{BNS}. For the reader's convenience, we include a proof here.

\begin{prop}\label{prop:OtherSpinc}
Let $Y^{\#},K^{\#},F^{\#}$ be as before. Then 
\[
HF^+(Y^{\#}_0(K^{\#}),\mathfrak T^{!}_{g^{\#}-2})=0.
\]
\end{prop}
\begin{proof}
Using the exact triangle 
\[
\xymatrix{
HF^+(\cdot) \ar[rr]^U &&HF^+(\cdot)\ar[ld]\\
 &\widehat{HF}(\cdot)\ar[lu] &,
}
\]
we only need to prove $\widehat{HF}(Y^{\#}_0(K^{\#}),\mathfrak T^{!}_{g^{\#}-2})=0$. 

It follows from Lemma~\ref{lem:Cable} and the K\"unneth formula that
\begin{equation}\label{eq:ZeroOutside}
\widehat{HFK}(Y^{\#},K^{\#},\mathfrak S^{!},[F^{\#}],i)=0, \quad \text{whenever }i\ge g^{\#}-2.
\end{equation}

For simplicity, we will write $C=C^{!}=CFK^{\infty}(Y^{\#},K^{\#},\mathfrak S^{!},[F^{\#}])$ for the rest of the proof.
Let 
\[
\widehat{A}_k=C\{\max\{i,j-k\}=0\},\quad \widehat B=C\{i=0\}.
\]
There are chain maps
\[
\widehat{v}_{k},\widehat{h}_{k}:\widehat{A}_k\to \widehat B
\] 
defined in \cite{OSzIntSurg}. Here $\widehat{v}_{k}$ first projects $\widehat{A}_k$ into $C\{i=0,j\le k\}$, then sends $C\{i=0,j\le k\}$ into $\widehat B$ by the inclusion map. The map $\widehat{h}_{k}$ is defined similarly. It first projects $\widehat{A}_k$ into $C\{i\le 0,j=k\}$, then sends $C\{i\le 0,j= k\}$ into $C\{j= k\}$, finally sends $C\{j= k\}$ to $\widehat B$ by a chain homotopy equivalence.

It is proved in \cites{OSzIntSurg,NiPropG} that $\widehat{HF}(Y^{\#}_0(K^{\#}),\mathfrak T^!_{k})$ is isomorphic to the homology of the mapping cone of 
\[\widehat{v}_{k}+\widehat{h}_{k}:\widehat{A}_k\to \widehat B.\]

Using (\ref{eq:ZeroOutside}), we conclude that $H_*(C\{i\le0,j=g^{\#}-2\})=0$, so $(\widehat{h}_{g^{\#}-2})_*=0$.
Hence $\widehat{HF}(Y^{\#}_0(K^{\#}),\mathfrak T^!_{g^{\#}-2})$ is isomorphic to the homology of the mapping cone of $\widehat{v}_{g^{\#}-2}$.
Using (\ref{eq:ZeroOutside}) again, we see that $(\widehat{v}_{g^{\#}-2})_*$ is an isomorphism, so the homology of its mapping cone is $0$. This finishes the proof.
\end{proof}

\begin{proof}[Proof of Theorem~\ref{thm:FPbound}]
By the K\"unneth formula, 
\begin{eqnarray*}
&&\widehat{HFK}(Y^{\#},K^{\#},[F^{\#}], g^{\#}-1)\\
&\cong& \quad\widehat{HFK}(Y,K,[F], g-1)\otimes \widehat{HFK}(Z,\overline{L'},[\overline E], g(L'))\\
&&\bigoplus \widehat{HFK}(Y,K,[F], g)\otimes \widehat{HFK}(Z,\overline{L'},[\overline E], g(L')-1).
\end{eqnarray*}
Since $\widehat{HFK}(Y,K,[F], g)\cong\widehat{HFK}(Z,\overline{L'},[\overline E], g(L'))\cong\mathbb F$,
we have
\begin{eqnarray*}
&&\mathrm{rank}\widehat{HFK}(Y^{\#},K^{\#},[F^{\#}], g^{\#}-1)\\
&=&\mathrm{rank}\widehat{HFK}(Y,K,[F], g-1)+\mathrm{rank}\widehat{HFK}(Z,\overline{L'},[\overline E], g(L')-1)\\
&=&r+1.
\end{eqnarray*}
 It follows from Propositions~\ref{prop:ZeroSurg} and \ref{prop:OtherSpinc} that
\[
\mathrm{rank}HF^+((Y^{\#})_0(K^{\#}),[\widehat{F^{\#}}], g^{\#}-2)=r-1.
\]

The manifold $(Y^{\#})_0(K^{\#})$ is a surface bundle over $S^1$. Its fiber $\widehat{F^{\#}}$ is a closed surface which is the union of $F$ and $\overline E$.
Let $\widehat{\sigma}$ be the monodromy, then $\widehat{\sigma}|F=\varphi$. Let $[\widehat{\sigma}]$ be the mapping class of $\widehat{\sigma}$.

As argued in \cite[Theorem~3.5]{BHS}, using work of Lee--Taubes \cite{LeeTaubes}, Kutluhan--Lee--Taubes \cite{KLT}, Kronheimer--Mrowka \cite{KMBook}, one sees that \[HF^+((Y^{\#})_0(K^{\#}),[\widehat{F^{\#}}],g^{\#}-2)\cong HF^{symp}_*(\widehat{F^{\#}},[\widehat{\sigma}]).\] So 
\[
\mathrm{rank}HF^{symp}_*(\widehat{F^{\#}},[\widehat{\sigma}])=r-1.
\]

By Corollary~\ref{cor:SympFPbound}, we have $N(\widehat{\sigma})\le r-1$. Since $\widehat{\sigma}|_{\overline E}$ has no fixed point, and $F$ is an essential surface in $\widehat{F^{\#}}$,
\[N(\widehat{\sigma}|_{F})=N(\widehat{\sigma})\le r-1.\]
Now our conclusion follows from Theorem~\ref{thm:JG}.
\end{proof}


\end{document}